\documentclass{amsart}

\usepackage{ragged2e}
\usepackage{a4wide}
\usepackage[utf8]{inputenc}
\usepackage{enumerate}

\usepackage{amsmath}
\usepackage{amsfonts}
\usepackage{mathabx}%\widecheck

\usepackage{bbm}
\usepackage{mathrsfs}
\usepackage{amsthm}
\usepackage{verbatim}
\usepackage{upgreek}
\usepackage{relsize}
\usepackage{dsfont}
\usepackage{graphicx}
\usepackage{amssymb}
\usepackage{xcolor}

\newcommand{\id}{\textnormal{id}}
\newcommand{\mc}{\mathcal}

\newcommand{\mrm}{\mathrm}

\newcommand{\mf}{\mathfrak}

\newcommand{\I}{\mathbbm{1}}

\newcommand{\md}{\operatorname{d}\!}
\newcommand{\cst}{\ifmmode \mathrm{C}^* \else $\mathrm{C}^*$\fi}

\newcommand{\NN}{\mathbb{N}}
\newcommand{\RR}{\mathbb{R}}

\newcommand{\GG}{\mathbb{G}}

\newcommand{\wot}{\ifmmode \textsc{wot} \else \textsc{wot}\fi}
\newcommand{\sot}{\ifmmode \textsc{sot} \else \textsc{sot}\fi}
\newcommand{\sots}{\ifmmode \textsc{sot}^* \else \textsc{sot}$^*$\fi}
\newcommand{\ssot}{\ifmmode \sigma\textsc{-sot} \else $\sigma$-\textsc{sot }\fi}
\newcommand{\ssots}{\ifmmode \sigma\textsc{-sot}^* \else $\sigma$-\textsc{sot }$^*$\fi}
\newcommand{\swot}{\ifmmode \sigma\textsc{-wot} \else $\sigma$-\textsc{wot}\fi}

\newcommand{\whG}{\widehat{\GG}}

\newcommand{\oon}{\operatorname}

\DeclareMathOperator{\Irr}{Irr}

\DeclareMathOperator{\Tr}{Tr}

\DeclareMathOperator{\M}{M}

\DeclareMathOperator{\LL}{L}

\DeclareMathOperator{\SU}{SU}

\DeclareMathOperator{\Inn}{Inn}
\DeclareMathOperator{\Aut}{Aut}

\newtheorem{theorem}{Theorem}[section]

\newtheorem{proposition}[theorem]{Proposition}

\theoremstyle{definition}

\numberwithin{equation}{section}

\begin{document}
\title{Scaling automorphisms of compact quantum groups}

\author{Jacek Krajczok}
\address{Vrije Universiteit Brussel\\
Pleinlaan 2\\
1050 Brussels\\
Belgium
}
\email{jacek.krajczok@vub.be}

\thanks{}


%46L67 Quantum groups (operator algebraic aspects)
%22D55 Kazhdan’s property (T), the Haagerup property, and generalizations
%43A30 Fourier and Fourier-Stieltjes transforms on nonabelian groups and on semigroups, etc.
%47L25 Operator spaces (= matricially normed spaces) [See also 46L07] in  Linear spaces and algebras of operators


\date{}

\begin{abstract}
In this short article we verify the conjecture proposed by P.~M.~Sołtan and the author, by proving that a second countable compact quantum group whose scaling automorphisms are inner, must be of Kac type.
\end{abstract}

\maketitle

\section{Introduction}

One of the features that distinguish the theory of compact quantum groups from its classical counterpart is the possibility that its Haar integral is non-tracial. This leads to a number of interesting phenomena, such as antipode being unbounded, discrete quantum groups being non-unimodular and taking contragradient representation not being involutive. One also obtains two canonical dynamics: modular automorphism group $(\sigma^h_t)_{t\in\RR}$ provided by the Tomita-Takesaki theory and scaling automorphism group $(\tau_t)_{t\in\RR}$. While some quantum groups are of Kac type, i.e.~have tracial Haar integral (equivalently the corresponding automorphism groups are trivial), there is plenty of non-Kac type examples: $q$-deformations, free unitary or free orthogonal quantum groups \cite{NeshveyevTuset}.

In recent work \cite{KSInvariants} P.~M.~Sołtan, together with the author, studied several invariants associated with a (locally) compact quantum group, canonically built from the modular data. Of particular importance turned out to be the invariant $T^{\tau}_{\Inn}(\GG)$, defined as the set of those $t\in\RR$, for which the scaling automorphism $\tau_t\in \Aut(\LL^{\infty}(\GG))$ is \emph{inner}. For example, it was used in \cite{KSExamples} to prove that for each $\lambda\in \left[0,1\right]$, there is an uncountably many compact quantum groups $\GG$ with $\LL^{\infty}(\GG)$ being an injective factor of type $\oon{III}_\lambda$. A similar invariant appeared in \cite{VaesOuter}.\\

The driving force in \cite{KSInvariants} was \cite[Conjecture 1]{KSInvariants}, stating that if $\GG$ is a second countable\footnote{I.e.~$\mrm{C}(\GG)$ is separable, see also \cite[Lemma 14.6]{KrajczokCoamenability}.} compact quantum group and $T^{\tau}_{\Inn}(\GG)=\RR$, then $\GG$ is of Kac type. An analogous statement holds true in the von Neumann algebra theory: if $\M$ is a von Neumann algebra with separable predual, then its Connes' invariant $T(\M)$ is trivial (all modular automorphisms are inner) if, and only if $\M$ is semifinite \cite[Proposition 27.2]{Stratila}.

Conjecture 1 in \cite{KSInvariants} was verified for several classes of quantum groups, including unitary quantum groups, quantum groups with $2$-dimensional representation $\pi$ satisfying $\dim_q(\pi)>2$ and compact quantum groups dual to type I discrete quantum groups. In this work we establish \cite[Conjecture 1]{KSInvariants} in full generality.\\

\textbf{Theorem A}. Let $\GG$ be a second countable compact quantum group, such that $T^{\tau}_{\oon{Inn}}(\GG)=\RR$. Then $\GG$ is of Kac type.\\

It was shown in \cite[Remark 3.1]{KSInvariants}, with a bicrossed product construction, that the second countability constraint is necessary. Moreover, an example of the quantum ``$az+b$'' group shows that outside the compact setting, it can happen that all scaling automorphisms are inner but nonetheless $\tau\neq \id$ \cite[Section 5.2]{KSInvariants}. Consequently these assumptions cannot be dropped.

Let us also mention that for a compact quantum group $\GG$ and a \emph{particular} time $t$, it can happen that $\tau_t$ is inner, but not implemented by any unitary in $\mrm{C}(\GG)$; in particular not implemented by a group-like unitary. This phenomenon occurs e.g.~for $\SU_q(3)$ \cite[Proposition 4.16]{KSInvariants}. Thus we cannot simply assume that our implementing unitaries are group-likes. A natural strategy that comes to mind, is to try to modify them to become group-like. While in general it is not clear how to perform such an operation, it can be done under the additional assumption that $\whG$ is $1$-i.c.c., i.e.~the relative commutant of $\Delta_{\GG}(\LL^{\infty}(\GG))$ in $\LL^\infty(\GG)\bar\otimes\LL^{\infty}(\GG)$ is trivial \cite[Theorem 6.5]{KSInvariants}.\\

\textbf{Acknowledgements}. This work was partially supported by FWO grant 1246624N. 

\textbf{AI disclosure}. An earlier version of this work was created with the help of ChatGPT 5.6 Sol.

\section{Proof of Theorem A}\label{sec:preliminaries}

We will freely use results from the theory of compact quantum groups \cite{NeshveyevTuset}, following mostly notational conventions from this book. \\

We will establish the following slightly stronger statement. It shows that it is the SOT-continuity of the implementation, rather than second countability of $\GG$, that is crucial.

\begin{proposition}\label{prop1}
Let $\GG$ be a compact quantum group. Assume that there is a group of unitary operators $(u_t)_{t\in \RR}$ in $\LL^{\infty}(\GG)$ which is continuous in SOT and $\tau_t=\oon{Ad}(u_t)|_{\LL^{\infty}(\GG)}$ for $t\in\RR$. Then $\GG$ is of Kac type.
\end{proposition}

Assuming this result, we first deduce Theorem A:

\begin{proof}[Proof of Theorem A]
We assume that each $\tau_t$ is inner and $\LL^2(\GG)$ is separable, hence the theorem of Kallman \cite{KallmanInner} applies, and there is a SOT-continuous group $(u_t)_{t\in \RR}$ of unitary operators in $\LL^{\infty}(\GG)$ which implement $(\tau_t)_{t\in\RR}$. Proposition \ref{prop1} ends the proof.
\end{proof}

\begin{proof}[Proof of Proposition \ref{prop1}]
Fix an arbitrary (class of) irreducible representation $\pi\in \Irr(\GG)$, choose its representative $U^{\pi}$ and an orthonormal basis in the corresponding Hilbert space $\mc{H}_\pi$ in which the operator $\uprho_\pi$ is diagonal with eigenvalues $\uprho_{\pi,1},\dotsc,\uprho_{\pi,\dim(\pi)}$. Recall that $\Tr(\uprho_\pi)=\Tr(\uprho_\pi^{-1})$ is the quantum dimension $\dim_q(\pi)$ of $\pi$. Our aim is to prove $\uprho_\pi=\I$: since $\pi$ is arbitrary, this will show that $\GG$ is of Kac type.

Stone's theorem \cite[Theorem 12.2]{SoltanPrimer}, together with functional calculus \cite[Section 10.2]{SoltanPrimer}, give us a strictly positive, self-adjoint operator $A$ affiliated with $\LL^{\infty}(\GG)$ such that $u_t=A^{it}\,(t\in\RR)$. By \cite[Lemma 3.2]{KSInvariants} we also know that $A^{it}$ is invariant under the modular automorphism group $\sigma^h$ of the Haar integral $h$ \cite{Stratila}, hence $A$ is affiliated with $\LL^\infty(\GG)^{h}=\{x\in\LL^{\infty}(\GG)\mid \sigma^h_s(x)=x\,(s\in\RR)\}$. For every $1\le i,j\le \dim(\pi), t\in\RR$ we have
\[
 A^{it}U^{\pi}_{i,j}A^{-it}=\tau_t(U^{\pi}_{i,j})=\bigl(\tfrac{\uprho_{\pi,i}}{\uprho_{\pi,j}}\bigr)^{it} U^{\pi}_{i,j}\quad\Rightarrow\quad 
A^{it} U^{\pi}_{i,j}=U^{\pi}_{i,j} \bigl(\tfrac{\uprho_{\pi,i}}{\uprho_{\pi,j}} A\bigr)^{it}
\]
and properties of Borel functional calculus imply that
\begin{equation}\label{eq2}
f(A)U^{\pi}_{i,j}=U^{\pi}_{i,j}\, f\circ T_{i,j}(A)\quad(f\in \mf{B}(\RR_{>0})),
\end{equation}
where $\mf{B}(\RR_{>0})$ is the $*$-algebra of bounded, Borel functions on $\RR_{>0}$ and $T_{i,j}\colon \RR_{>0}\ni x\mapsto \tfrac{\uprho_{\pi,i}}{\uprho_{\pi,j}} x\in \RR_{>0}$. Fix $f\in \mf{B}(\RR_{>0})$. Using equation \eqref{eq2}, property that $A$ is affiliated with $\LL^{\infty}(\GG)^h$, $U^{\pi}$ is unitary and the formula $\sigma^h_t(U^{\pi}_{i,j})=(\uprho_{\pi,i}\uprho_{\pi,j})^{it} U^{\pi}_{i,j}$ valid for any $t,i,j$, we calculate
\begin{equation}\begin{split}\label{eq1}
&\quad\;
\sum_{i,j=1}^{\dim(\pi)}
\tfrac{\uprho_{\pi,j}}{\dim_q(\pi)}h\bigl(f\circ T_{i,j}(A)U^{\pi *}_{i,j} U^{\pi}_{i,j}\bigr)=
\sum_{i,j=1}^{\dim(\pi)}
\tfrac{\uprho_{\pi,j}}{\dim_q(\pi)}h\bigl( U^{\pi}_{i,j}
f\circ T_{i,j}(A)\sigma^h_{-i}(U^{\pi *}_{i,j})\bigr)\\
&=
\sum_{i,j=1}^{\dim(\pi)}
\tfrac{1}{\uprho_{\pi,i}\dim_q(\pi)}h\bigl( U^{\pi}_{i,j}
f\circ T_{i,j}(A) U^{\pi *}_{i,j}\bigr)=
\sum_{i,j=1}^{\dim(\pi)}
\tfrac{1}{\uprho_{\pi,i}\dim_q(\pi)}h\bigl( f (A) U^{\pi}_{i,j}
 U^{\pi *}_{i,j}\bigr)\\
 &=
 \sum_{i=1}^{\dim(\pi)}
\tfrac{1}{\uprho_{\pi,i}\dim_q(\pi)}h( f (A)  )=h(f(A)).
\end{split}\end{equation}
On the other hand,
\begin{equation}\label{eq3}
\sum_{i,j=1}^{\dim(\pi)}\tfrac{\uprho_{\pi,j}}{\dim_q(\pi)}
h\bigl(f(A)U^{\pi *}_{i,j} U^{\pi}_{i,j}\bigr)=
\sum_{ j=1}^{\dim(\pi)}\tfrac{\uprho_{\pi,j}}{\dim_q(\pi)}
h(f(A) )=h(f(A)).
\end{equation}
For $k\in \NN$, consider function
\[
X_k\in\mf{B}(\RR_{>0})\colon \quad X_k(t)=\begin{cases}
\log(k^{-1}), & 0<t\le k^{-1},\\
\log(t), & k^{-1} < t \le k,\\
\log(k), & k< t.
\end{cases} 
\]
Combining \eqref{eq1} and \eqref{eq3} we obtain
\begin{equation}\label{eq4}
\sum_{i,j=1}^{\dim(\pi)} \tfrac{\uprho_{\pi,j}}{\dim_q(\pi)} h\bigl( (X_k-X_k\circ T_{i,j})(A) U^{\pi *}_{i,j} U^{\pi}_{i,j}\bigr)=0\quad(k\in\NN).
\end{equation}
Next, observe that
\[
X_k(t)=\log(k^{-1}) + \int_0^{t}Y_k(s)\md s\;(t\in\RR_{>0})\quad\textnormal{where}\quad 
Y_k\in \mf{B}(\RR_{>0})\colon\quad Y_k(s)=\begin{cases}
0, & 0<s\le k^{-1},\\
\tfrac{1}{s}, & k^{-1} < s \le k,\\
0, & k< s.
\end{cases}
\]
With this description, we easily obtain
\[
\bigl|X_k(t)-X_k\circ T_{i,j}(t)\bigr|\le 
\Bigl| \int_{\tfrac{\uprho_{\pi,i}}{\uprho_{\pi,j}}t }^{t} \tfrac{1}{s}\md s\Bigr|=
\bigl|\log(t) - \log\bigl(\tfrac{\uprho_{\pi,i}}{\uprho_{\pi,j}} t\bigr)\bigr|=
\bigl| \log\bigl(\tfrac{\uprho_{\pi,i}}{\uprho_{\pi,j}}  \bigr)\bigr|\quad(k\in\NN,t\in\RR_{>0}),
\]
in particular the family of functions $(X_k-X_k\circ T_{i,j})_{k\in\NN}$ is uniformly bounded and converges pointwise to the constant function $-\log\bigl(\tfrac{\uprho_{\pi,i}}{\uprho_{\pi,j}}\bigr)\I$. Properties of Borel functional calculus \cite[Theorem 10.5]{SoltanPrimer}, normality of $h$, orthogonality relations and equality \eqref{eq4} imply
\[\begin{split}
0&=\lim_{k\to\infty}
\sum_{i,j=1}^{\dim(\pi)} \tfrac{\uprho_{\pi,j}}{\dim_q(\pi)} h\bigl( (X_k-X_k\circ T_{i,j})(A) U^{\pi *}_{i,j} U^{\pi}_{i,j}\bigr)=
\sum_{i,j=1}^{\dim(\pi)} \tfrac{-\uprho_{\pi,j}}{\dim_q(\pi)} \log\bigl(\tfrac{\uprho_{\pi,i}}{\uprho_{\pi,j}}\bigr)
h\bigl( U^{\pi *}_{i,j} U^{\pi}_{i,j}\bigr)\\
&=
\sum_{i,j=1}^{\dim(\pi)} \tfrac{-\uprho_{\pi,j}}{\uprho_{\pi,i}\dim_q(\pi)^2} \log\bigl(\tfrac{\uprho_{\pi,i}}{\uprho_{\pi,j}}\bigr)=
\tfrac{-1}{\dim_q(\pi)}\sum_{i=1}^{\dim(\pi)}\bigl( \tfrac{1}{\uprho_{\pi,i}}-\uprho_{\pi,i}\bigr)\log(\uprho_{\pi,i}).
\end{split}\]
Since the function $g\colon \RR_{>0}\ni t\mapsto (t^{-1}-t)\log(t)\in \RR$ satisfies $g\le 0$ and $g(t)=0\Leftrightarrow t=1$, the above equation gives $\uprho_{\pi,i}=1$ for all $1\le i\le \dim(\pi)$ and ends the proof.
\end{proof}

\bibliographystyle{plain}
\bibliography{bibliografia}

\end{document}